\documentclass{amsart}

\usepackage{graphicx}
\usepackage{amssymb}
\usepackage[arrow, matrix, curve]{xy}
\usepackage{nicefrac}
\usepackage{hyperref}

\usepackage[margin=3cm]{geometry}
\usepackage[onehalfspacing]{setspace}

\newtheorem{theorem}{Theorem}[section]
\newtheorem*{maintheorem}{Main Theorem}
\newtheorem{proposition}[theorem]{Proposition}
\newtheorem{lemma}[theorem]{Lemma}
\newtheorem{corollary}[theorem]{Corollary}

\theoremstyle{definition}

\newtheorem{example}[theorem]{Example}

\theoremstyle{remark}
\newtheorem*{remark}{Remark}

\numberwithin{equation}{section}

\begin{document}

\title{The Farrell-Jones Conjecture for fundamental groups of graphs of abelian groups}

\author{Giovanni Gandini}
\address{K{\o}benhavns Universitet, Institut for Matematiske Fag, Universitetsparken 5, 2100 K{\o}benhavn \O, Denmark}
\email{ggandini@math.ku.dk}
\urladdr{http://www.math.ku.dk/~zjb179}

\author{Sebastian Meinert}
\address{Freie Universit\"at Berlin, Institut f\"ur Mathematik, Arnimallee 7, 14195 Berlin, Germany}
\email{sebastian.meinert@fu-berlin.de}
\urladdr{http://userpage.fu-berlin.de/meinert}

\author{Henrik R\"uping}
\address{Rheinische Friedrich-Wilhelms-Universit\"at Bonn, Mathematisches Institut, Endenicher Allee 60, 53115 Bonn, Germany}
\email{henrik.rueping@hcm.uni-bonn.de}
\urladdr{http://www.math.uni-bonn.de/people/rueping}

\subjclass[2010]{Primary 18F25; Secondary 19A31, 19B28, 19G24}
\keywords{Farrell-Jones Conjecture, algebraic K- and L-theory of group rings, fundamental groups of graphs of abelian groups, generalized Baumslag-Solitar groups, Baumslag-Solitar groups}

\begin{abstract}
We show that the Farrell-Jones Conjecture holds for fundamental groups of graphs of groups with abelian vertex groups. As a special case, this shows that the conjecture holds for generalized Baumslag-Solitar groups.
\end{abstract}

\maketitle

\section{Introduction}

Denote by $\mathfrak{C}$ the class of groups that satisfy the K- and L-theoretic Farrell-Jones Conjecture with finite wreath products (with coefficients in additive categories) with respect to the family of virtually cyclic subgroups. Farrell and Wu \cite{FW13} showed that the conjecture holds for the solvable Baumslag-Solitar groups, and Wegner \cite{W13} generalized their proof to show that the conjecture in fact holds for all solvable groups.

Let $(\Gamma,\mathcal{G})$ be a finite graph of finitely generated abelian groups. We construct a group homomorphism $\phi$ from $\pi_1(\Gamma,\mathcal{G})$ to a semidirect product $\mathbb{Q}^m \rtimes F_n$, where $F_n$ denotes the free group of rank $n$. Wegner's result implies that $\mathbb{Q}^m\rtimes F_n$ lies in $\mathfrak{C}$ (Corollary~\ref{corollary: rational-by-free groups}). Then, given a torsion-free cyclic subgroup $C\leq\mathbb{Q}^m\rtimes F_n$ we show that its preimage $\phi^{-1}(C)\leq\pi_1(\Gamma,\mathcal{G})$ is a directed colimit of CAT(0)-groups and hence lies in $\mathfrak{C}$. Together with inheritance properties of $\mathfrak{C}$ (Theorem~\ref{theorem: inheritance properties}) and a sequence of colimit arguments this proves the following:

\begin{maintheorem}
Let $(\Gamma,\mathcal{G})$ be a graph of abelian groups. Then $\pi_1(\Gamma,\mathcal{G})$ lies in $\mathfrak{C}$.
\end{maintheorem}

Here, we do not require $\Gamma$ to be finite or countable, and we do not make any assumptions on the cardinality of the generating sets of the groups of $\mathcal{G}$.

A \emph{generalized Baumslag-Solitar group} is the fundamental group of a finite graph of infinite cyclic groups.

\begin{corollary}
All generalized Baumslag-Solitar groups, and in particular all Baumslag-Solitar groups, lie in $\mathfrak{C}$.
\end{corollary}

\begin{remark}
Farrell and Wu have informed us about a recent independent result of theirs which proves the Farrell-Jones Conjecture for Baumslag-Solitar groups.
\end{remark}

\subsubsection*{Acknowledgements}The first author was supported by the Danish National Research Foundation (DNRF) through the Centre for Symmetry and Deformation. The second author was supported by the Deutsche Forschungsgemeinschaft (DFG) through the Berlin Mathematical School (BMS). The third author was supported by the Leibniz Prize of Prof. Dr. Wolfgang L{\"u}ck granted by the Deutsche Forschungsgemeinschaft.

We would like to thank Dieter Degrijse for his helpful comments.

\section{Facts about the Farrell-Jones Conjecture}

The proof of the main theorem will rely on previously known cases of the conjecture and on inheritance properties. The following list is not complete; it highlights results that will be made use of in the present paper.

\begin{theorem}[\cite{LR04},\cite{BFL11},\cite{BL12},\cite{W12},\cite{W13}]
\label{theorem: inheritance properties}
The class $\mathfrak{C}$ has the following properties:
\begin{enumerate}
\item\label{inheritance: CAT(0)} CAT(0)-groups lie in $\mathfrak{C}$.
\item\label{inheritance: solvable} Virtually solvable groups lie in $\mathfrak{C}$.
\item\label{inheritance: subgroups} The class $\mathfrak{C}$ is closed under taking subgroups.
\item\label{inheritance: colimits} The class $\mathfrak{C}$ is closed under taking directed colimits.
\item\label{inheritance: homomorphism} Let $f\colon G\to H$ be a group homomorphism and assume that $H$ lies in $\mathfrak{C}$ and $f^{-1}(C)$ lies in $\mathfrak{C}$ for any torsion-free cyclic subgroup $C$ of $H$. Then $G$ lies in $\mathfrak{C}$.
\end{enumerate}
\end{theorem}

\begin{proof}
For CAT(0)-groups, the version with finite wreath products follows from the version without wreath products, since the wreath product of a CAT(0)-group with a finite group is again a CAT(0)-group. 
The wreath product version for virtually solvable groups has been shown in \cite[Theorem~1.1]{W13}.

The inheritance properties of the version with finite wreath products follow easily from the inheritance properties of the version without wreath products; see for example \cite[Proposition~2.17]{W13} and \cite[Proposition~2.22]{W13}.
\end{proof}

We will individually refer to each property above as ``property ($\ast$)''.

\begin{corollary}\label{corollary: rational-by-free groups} If $G$ is a solvable group and $H\in\mathfrak{C}$ then any semidirect product $G\rtimes H$ lies in $\mathfrak{C}$.
\end{corollary}

\begin{proof}
Consider the projection homomorphism $G\rtimes H\to H$. The claim follows from properties (\ref{inheritance: solvable}) and (\ref{inheritance: homomorphism}), as preimages of cyclic subgroups of $H$ are solvable.
\end{proof}

In particular, as finitely generated free groups are CAT(0)-groups, for any $m,n\in\mathbb{N}$ any semidirect product $\mathbb{Q}^m\rtimes F_n$ lies in $\mathfrak{C}$.

\section{Graphs of groups}

Given a connected graph $\Gamma$ (in the sense of Serre) and an oriented edge $e\in E(\Gamma)$, denote by $\iota(e)\in V(\Gamma)$ its initial and by $\tau(e)\in V(\Gamma)$ its terminal vertex. If by $\overline{e}$ we denote the edge $e$ with opposite orientation then $\iota(\overline{e})=\tau(e)$ and $\tau(\overline{e})=\iota(e)$. A \emph{graph of groups structure} $\mathcal{G}$ on $\Gamma$ consists of families of groups $(G_v)_{v\in V(\Gamma)}$ and $(G_e)_{e\in E(\Gamma)}$ satisfying $G_{\overline{e}}=G_e$ for all $e\in E(\Gamma)$ together with an injective group homomorphism $\alpha_e\colon G_e\hookrightarrow G_{\iota(e)}$ for each $e\in E(\Gamma)$. We call the pair $(\Gamma,\mathcal{G})$ a \emph{graph of groups}.

Given a maximal tree $T$ in $\Gamma$, let $\pi_1(\Gamma,\mathcal{G},T)$ be the group generated by the groups $G_v,v\in V(\Gamma)$ and the elements $e\in E(\Gamma)$ subject to the relations
\begin{itemize}
\item[(i)] $\overline{e}=e^{-1}$ for all $e\in E(\Gamma)$;
\item[(ii)] $e\cdot\alpha_{\overline{e}}(s)\cdot\overline{e}=\alpha_e(s)$ for all $e\in E(\Gamma)$ and $s\in G_e$;
\item[(iii)] $e=1$ if $e\in E(T)$.
\end{itemize}
We call $\pi_1(\Gamma,\mathcal{G},T)$ the \emph{fundamental group} of $(\Gamma,\mathcal{G})$ relative to $T$. For each $v\in V(\Gamma)$ the canonical map $G_v\to\pi_1(\Gamma,\mathcal{G},T)$ turns out to be injective \cite[Corollary~1.9]{B93}. The isomorphism type of $\pi_1(\Gamma,\mathcal{G},T)$ does not depend on the choice of $T$ \cite[Proposition~20]{S}, and we will often speak of \emph{the} fundamental group of $(\Gamma,\mathcal{G})$ and denote it by $\pi_1(\Gamma,\mathcal{G})$.

\begin{example}
Let $G=BS(p,q)=\langle x,t\ |\ tx^pt^{-1}=x^q \rangle$. Then $G$ is isomorphic to the fundamental group of a graph of groups with a single edge $e$ and vertex $v$, where $G_v=G_e=\langle x\rangle\cong\mathbb{Z}$ and $\alpha_{\overline{e}}=(x\mapsto x^p)$ and $\alpha_{e}=(x\mapsto x^q)$.
\end{example}

A \emph{subgraph of subgroups} of a graph of groups $(\Gamma,\mathcal{G})$ is a graph of groups $(\Gamma',\mathcal{G}')$ such that $\Gamma'\subseteq\Gamma$, for all $v\in V(\Gamma')$ and $e\in E(\Gamma')$ we have $G'_v\leq G_v$ and $G'_e\leq G_e$ respectively, and $\alpha'_e={\alpha_e}_{|G'_e}$ for all $e\in E(\Gamma')$. If $T'$ and $T$ are maximal trees in $\Gamma'$ and $\Gamma$ respectively such that $T'\subseteq T$ then there is a natural group homomorphism $\pi_1(\Gamma',\mathcal{G}',T')\to\pi_1(\Gamma,\mathcal{G},T)$ that maps for $v\in V(\Gamma')$ every $x\in G'_v$ to $x\in G_v$ and every $e\in E(\Gamma')$ to $e\in E(\Gamma)$.

\begin{lemma}
\label{lemma: fundamental group and colimit}
Let $(\Gamma_i,\mathcal{G}_i)_{i\in I}$ be a directed system of graphs of groups with binary relation $\subseteq$ where $(\Gamma_i,\mathcal{G}_i)\subseteq (\Gamma_j,\mathcal{G}_j)$ if $(\Gamma_i,\mathcal{G}_i)$ is a subgraph of subgroups of $(\Gamma_j,\mathcal{G}_j)$. Moreover, let $(T_i)_{i\in I}$ be a directed system of corresponding maximal trees, i.e. $T_i$ is a maximal tree in $\Gamma_i$ for all $i\in I$ and $T_i\subseteq T_j$ whenever $(\Gamma_i,\mathcal{G}_i)\subseteq (\Gamma_j,\mathcal{G}_j)$. Let
\begin{itemize}
\item $\Gamma=\bigcup_{i\in I}{\Gamma_i}$;
\item $\mathcal{G}=\bigcup_{i\in I}{\mathcal{G}_i}$ be the graph of groups structure on $\Gamma$ whose vertex and edge groups are the unions of the vertex and edge groups of the $\mathcal{G}_i$'s and where for $e\in E(\Gamma)$ and $s\in G_e$ we define $\alpha_e(s)$ by $(\alpha_e)_i(s)\in (G_{\iota(e)})_i\leq G_{\iota(e)}$ whenever $e\in E(\Gamma_i)$ and $s\in (G_e)_i$;
\item $T=\bigcup_{i\in I}{T_i}$ be our choice of a maximal tree in $\Gamma$.
\end{itemize}
Consider the directed system of fundamental groups of graphs of groups defined by the natural maps $\pi_1(\Gamma_i,\mathcal{G}_i,T_i)\to\pi_1(\Gamma_j,\mathcal{G}_j,T_j)$ whenever $(\Gamma_i,\mathcal{G}_i,T_i)\subseteq (\Gamma_j,\mathcal{G}_j,T_j)$. We have\[\pi_1(\Gamma,\mathcal{G},T)\cong\mathrm{colim}_{i\in I}\ {\pi_1(\Gamma_i,\mathcal{G}_i,T_i)}.\]
\end{lemma}

\begin{proof}
It easily follows from the definition of $(\Gamma,\mathcal{G},T)$ that $\pi_1(\Gamma,\mathcal{G},T)$ has the universal property of $\mathrm{colim}_{i\in I}\ {\pi_1(\Gamma_i,\mathcal{G}_i,T_i)}$, whence the claim.
\end{proof}

Given a graph of groups $(\Gamma,\mathcal{G})$, one can define a simplicial tree $X=\widetilde{(\Gamma,\mathcal{G})}$, the \emph{Bass-Serre covering} of $(\Gamma,\mathcal{G})$, and a continuous map $p\colon X\to\Gamma$ sending edges to edges such that the group $\pi_1(\Gamma,\mathcal{G})$ acts on $X$ by simplicial automorphisms without edge inversions and the stabilizer of $v\in V(X)$ is conjugate to the vertex group $G_{p(v)}\in\mathcal{G}$. Vice versa, by the fundamental theorem of Bass-Serre theory \cite[section~I.5.3]{S} any action of a group $G$ on a simplicial tree $T$ gives rise to a (generally non-canonical) graph of groups structure $\mathcal{G}$ on the quotient graph $G\backslash T$ such that $\pi_1(G\backslash T,\mathcal{G})\cong G$.

\begin{lemma}
\label{lemma: acts on tree with finite stabilizers}
If a group acts on a simplicial tree with finite point stabilizers then it lies in $\mathfrak{C}$.
\end{lemma}

\begin{proof}
A group that acts on a simplicial tree with finite point stabilizers is isomorphic to the fundamental group of a graph of finite groups $(\Gamma,\mathcal{G})$. By Lemma~\ref{lemma: fundamental group and colimit}, the group $\pi_1(\Gamma,\mathcal{G})$ is isomorphic to the colimit of the directed system of fundamental groups associated to the directed system of finite subgraphs of subgroups of $(\Gamma,\mathcal{G})$. Fundamental groups of finite graphs of finite groups are CAT(0)-groups (in fact, they are virtually finitely generated free) and hence lie in $\mathfrak{C}$ by property~(\ref{inheritance: CAT(0)}). Therefore, $\pi_1(\Gamma,\mathcal{G})$ is isomorphic to the colimit of a directed system of groups that lie in $\mathfrak{C}$ and hence lies in $\mathfrak{C}$ by property~(\ref{inheritance: colimits}).
\end{proof}

A \emph{tree of groups} is a graph of groups whose underlying graph is a tree.

\begin{proposition}
\label{prop:FundTreeAbCat}
The fundamental group of a finite tree of finitely generated abelian groups $(T,\mathcal{G})$ is a CAT(0)-group.
\end{proposition}

We will make use of the following theorem:

\begin{theorem}[Equivariant Gluing {\cite[II.11.18]{BH}}]\label{thm:eqgluing} Let $\Gamma_0$, $\Gamma_1$ and $H$ be groups acting properly
by isometries on complete CAT(0) spaces $X_0$, $X_1$ and $Y$ respectively. Suppose that for
$j = 0, 1$ there exists an injective group homomorphism $\varphi_j\colon H \rightarrow \Gamma_j$ and a $\varphi_j$-equivariant isometric
embedding $f_j\colon Y \rightarrow X_j$. Then\begin{enumerate}
\item the amalgamated free product $\Gamma = \Gamma_0 \ast_H \Gamma_1$ associated to the maps $\varphi_j$ acts properly by isometries on a complete CAT(0) space $X$;
\item if the given actions of $\Gamma_0$, $\Gamma_1$ and $H$ are cocompact, the action of $\Gamma$ on $X$ is cocompact.
\end{enumerate}
\end{theorem}

We also need that the spaces $X_0$ and $X_1$ embed equivariantly and isometrically into $X$. However, this is clear from the construction given in the proof of the Equivariant Gluing Theorem in \cite{BH}.

\begin{proof}[Proof of Proposition~\ref{prop:FundTreeAbCat}]

Define for $v\in V(T)$ and $e\in E(T)$ the $\mathbb{R}$-vector spaces $X_v=G_v\otimes_\mathbb{Z}\mathbb{R}$ and $X_e=G_e\otimes_\mathbb{Z}\mathbb{R}$ respectively. The induced action of $G_v$ on $X_v$ given by
\[G_v\times X_v\to X_v,\ (g,x)\mapsto (g\!\otimes\! 1)+x\]
is proper and cocompact, and we analogously obtain a proper and cocompact action of $G_e$ on $X_e$. Let $v_0\in V(T)$ and exhaust the finite tree $T$ by subtrees $\left\{v_0\right\}=T_0\subset \ldots\subset T_n=T$ such that for all $i=1,\ldots,n$ the tree $T_i$ has one more vertex $v_i$ than $T_{i-1}$. We will denote the graph of groups structure on $T_i$ obtained by restricting $\mathcal{G}$ to $T_i\subseteq T$ also by $\mathcal{G}$. For each $i=1,\ldots,n$  denote by $e_i$ the unique oriented edge of $T_i$ for which $\iota(e_i)\in V(T_{i-1})$ and $\tau(e_i)=v_i$.

Choose an inner product on the finite-dimensional $\mathbb{R}$-vector space $X_{v_0}$ and thereby endow it with a complete CAT(0) metric. Independent of this choice, $G_{v_0}$ acts on $X_{v_0}$ by isometries. We inductively construct for each $i=1,\ldots,n$ inner products on $X_{e_i}$ and $X_{v_i}$ such that the $\alpha_{e_i}$-equivariant respectively $\alpha_{\overline{e_i}}$-equivariant embeddings $X_{\iota(e_i)}\hookleftarrow X_{e_i}\hookrightarrow X_{v_i}$ induced by the injective edge homomorphisms $\pi_1(T_{i-1},\mathcal{G})\geq G_{\iota(e_i)}\stackrel{\alpha_{e_i}}{\longleftarrow}G_{e_i}\stackrel{\alpha_{\overline{e_i}}}{\longrightarrow} G_{v_i}$ are isometric. In order to do so, pull back the inner product on $X_{\iota(e_i)}$ to obtain an inner product on $X_{e_i}$. Then, choose any inner product on $X_{v_i}$ that extends the inner product on $X_{e_i}\hookrightarrow X_{v_i}$.

By applying Theorem \ref{thm:eqgluing} repeatedly, we construct for $i=1,\ldots,n$ a complete CAT(0) space $X_{T_i}$ on which $\pi_1(T_i,\mathcal{G})$ acts properly and cocompactly by isometries, and into which for $j\le i$ each $X_{v_j}$ embeds equivariantly and isometrically.
\end{proof}

\begin{remark} Free products with amalgamation of virtually finitely generated abelian groups need not be CAT(0)-groups; a counterexample can be found in \cite[III.$\Gamma$.6.13]{BH}. However, if the amalgam is virtually cyclic then the vertex groups can be arbitrary CAT(0)-groups \cite[Corollary~II.11.19]{BH}.
\end{remark}

\begin{corollary}
\label{corollary: tree of groups}
The fundamental group of a tree of finitely generated abelian groups lies in $\mathfrak{C}$.
\end{corollary}

\begin{proof}
Any graph of groups can be exhausted by the directed system of its finite subgraphs of groups. The claim follows from Lemma~\ref{lemma: fundamental group and colimit}, Proposition~\ref{prop:FundTreeAbCat}, and properties (\ref{inheritance: CAT(0)}) and (\ref{inheritance: colimits}).
\end{proof} 

\section{Proof of the main theorem}

\begin{proof}[Proof of the Main Theorem]
We may assume that $(\Gamma,\mathcal{G})$ is a finite graph of finitely generated abelian groups; this follows from three consecutive applications of Lemma~\ref{lemma: fundamental group and colimit}:
\begin{enumerate}
\item Let $(\Gamma,\mathcal{G})$ be a finite graph of abelian groups with finitely generated edge groups. For every vertex $v\in V(\Gamma)$ there exists a finitely generated subgroup of $G_v$ that contains the images of all adjacent edge homomorphisms so that we can easily find a directed system of finite subgraphs of finitely generated subgroups of $(\Gamma,\mathcal{G})$ that exhausts $(\Gamma,\mathcal{G})$.
\item If $(\Gamma,\mathcal{G})$ is a finite graph of abelian groups, we can write every edge group as the directed colimit of its finitely generated subgroups and $\pi_1(\Gamma,\mathcal{G})$ as the directed colimit of fundamental groups of finite graphs of abelian groups with finitely generated edge groups.
\item Finally, any graph of (abelian) groups can be exhausted by the directed system of its finite subgraphs of (abelian) groups.
\end{enumerate}
Let $T$ be a maximal tree in $\Gamma$. We will construct a group homomorphism $\phi$ from $\pi_1(\Gamma,\mathcal{G},T)$ to a group of the form $\mathbb{Q}^m\rtimes F_n$, where $\mathbb{Q}^m\rtimes F_n$ lies in $\mathfrak{C}$ by Corollary~\ref{corollary: rational-by-free groups}. We then prove that $\pi_1(\Gamma,\mathcal{G},T)$ lies in $\mathfrak{C}$ by showing that all preimages of torsion-free cyclic subgroups of $\mathbb{Q}^m\rtimes F_n$ lie in $\mathfrak{C}$, i.e. by applying property~(\ref{inheritance: homomorphism}).

Let the vertex set of $\Gamma$ be given by $\left\{v_1,\ldots,v_k\right\}$ and define $Q$ as the $\mathbb{Q}$-vector space \[Q=\bigoplus_{i=1}^{k}\left(G_{v_i}\otimes_\mathbb{Z}\mathbb{Q}\right).\] For every $e\in E(\Gamma)$ the injective group homomorphism $\alpha_e\colon G_e\to G_{\iota(e)}$ gives rise to an injective $\mathbb{Q}$-linear homomorphism $M_e=\alpha_e\otimes_\mathbb{Z}id\colon G_e\otimes_\mathbb{Z}\mathbb{Q}\to G_{\iota(e)}\otimes_\mathbb{Z}\mathbb{Q}$. Define $R$ to be the $\mathbb{Q}$-subvector space of $Q$ spanned by the vectors \[M_{\overline{e}}(s\!\otimes\! 1) - M_e(s\!\otimes\! 1)\text{ for all }e\in E(T),\ s\in G_e.\]
For every vertex $v\in V(\Gamma)$ the rationalized vertex group $G_v\otimes_\mathbb{Z}\mathbb{Q}$ naturally embeds into $Q/R$, which can be seen as follows: Fix an orientation $\mathcal{O}(T)\subset E(T)$ for each edge of $T$ and suppose that this is not the case, i.e. we can find an element $0\neq q\in G_v\otimes_\mathbb{Z}\mathbb{Q}$ and for every $e\in\mathcal{O}(T)$ an element $q_e\in G_e\otimes_\mathbb{Z}\mathbb{Q}$ such that
\begin{equation}
q = \sum_{e \in\mathcal{O}(T)}(M_{\overline{e}}(q_e) - M_e(q_e))\in Q.
\label{eq:linearComb}
\end{equation}
Consider the subforest $F\subseteq T$ spanned by all edges for which $q_e\neq 0$. It contains at least one edge, as the right hand side of (\ref{eq:linearComb}) would otherwise be zero, contradicting that $q\neq 0$. Choose a leaf $w\in V(F)$, i.e. a vertex of valence 1, such that $w\neq v$ and let $e$ be the unique edge in $\mathcal{O}(T)\cap E(F)$ adjacent to $w$, say with $\iota(e)=w$. Since $q\in G_v\otimes_\mathbb{Z}\mathbb{Q}$ with $v\neq w$ and $e$ is the only edge in $\mathcal{O}(T)\cap E(F)$ adjacent to $w$, projecting (\ref{eq:linearComb}) to the factor $G_w\otimes_\mathbb{Z}\mathbb{Q}\leq Q$ gives rise to the equation \[0= -M_e(q_e).\]
However, this is a contradiction, as $q_e\neq 0$ and $M_e$ is injective.

Let $\left\{e_1,\ldots,e_n\right\}$ be the set of edges of $\Gamma\smallsetminus T$ and denote by $F_n$ the free group with basis $\left\{e_1,\ldots,e_n\right\}$. We obtain a linear representation $\rho\colon F_n\to GL(Q/R)$ by extending for every $e_i$ the isomorphism of subspaces $M_{e_i}\circ {M_{\overline{e_i}}}^{-1}\colon \mathrm{im}(M_{\overline{e_i}})\stackrel{\cong}{\to} \mathrm{im}(M_{e_i})$ to an automorphism of the finite-dimensional $\mathbb{Q}$-vector space $Q/R$. Define a group homomorphism $$\phi\colon\pi_1(\Gamma,\mathcal{G},T)\to (Q/R)\rtimes_\rho F_n$$ by mapping for $v\in V(\Gamma)$ any element $x\in G_v$ to $(x\!\otimes\! 1,1)$ and $e\in E(\Gamma\smallsetminus T)$ to $(0\!\otimes\! 0,e)$. This assignment is well-defined: Suppose that $x\in G_v$ lies in the image of $\alpha_e$ for some edge $e\in E(\Gamma)$ with $\iota(e)=v$, i.e. $x=\alpha_e(s)$ for some $s\in G_e$. Then $\phi(x)=\phi(\alpha_e(s))=(\alpha_{e}(s)\!\otimes\! 1,1)$. By definition, in $\pi_1(\Gamma,\mathcal{G},T)$ we have that
\[x=\begin{cases}\alpha_{\overline{e}}(s) & \text{if }e\in E(T),\text{ and}\\
e\cdot\alpha_{\overline{e}}(s)\cdot\overline{e} & \text{if }e\in E(\Gamma\smallsetminus T).
\end{cases}\]
In the first case, $\phi(\alpha_{\overline{e}}(s))=(\alpha_{\overline{e}}(s)\!\otimes\! 1,1)$, where $\alpha_{\overline{e}}(s)\!\otimes\! 1=\alpha_e(s)\!\otimes\! 1$ in $Q/R$ and hence $\phi(\alpha_{\overline{e}}(s))=\phi(x)$. In the second case,
\begin{align*}
\phi(e\cdot\alpha_{\overline{e}}(s)\cdot\overline{e}) &= (0\!\otimes\! 0,e)\cdot (\alpha_{\overline{e}}(s)\!\otimes\! 1,1)\cdot (0\!\otimes\! 0,\overline{e})\\
&= (0\!\otimes\! 0 + \rho(e)(\alpha_{\overline{e}}(s)\!\otimes\! 1) + 0\!\otimes\! 0,e\overline{e})\\
&= (M_{e}({M_{\overline{e}}}^{-1}(\alpha_{\overline{e}}(s)\!\otimes\! 1)),e\overline{e})\\
&= (\alpha_e(s)\!\otimes\! 1,1)=\phi(x)
\end{align*}
whence $\phi$ is well-defined. Recall that $(Q/R)\rtimes_\rho F_n$ lies in $\mathfrak{C}$ by Corollary~\ref{corollary: rational-by-free groups}, $Q/R$ being isomorphic to $\mathbb{Q}^m$ for some $m\in\mathbb{N}$. 

Let $C$ be a torsion-free cyclic subgroup of $(Q/R)\rtimes_\rho F_n$ and first assume that $C$ is not contained in $(Q/R)\rtimes\left\{1\right\}$. Consider the induced subgroup action of $\phi^{-1}(C)\leq\pi_1(\Gamma,\mathcal{G},T)$ on the Bass-Serre covering tree $X=\widetilde{(\Gamma,\mathcal{G})}$ and recall that every point stabilizer of this action is contained in a conjugate of some vertex group $G_v,\ v\in V(\Gamma)$. For each vertex $v\in V(\Gamma)$ the kernel of the natural map $G_v\to G_v\otimes_\mathbb{Z}\mathbb{Q}$ equals the torsion subgroup of $G_v$, and $G_v\otimes_\mathbb{Z}\mathbb{Q}$ embeds into $Q/R$ and hence into the normal subgroup $(Q/R)\rtimes\left\{1\right\}$. Consequently, $\phi^{-1}(C)$ contains of every point stabilizer only its torsion subgroup and acts on $X$ with finite point stabilizers. We conclude that $C$ lies in $\mathfrak{C}$ by Lemma~\ref{lemma: acts on tree with finite stabilizers}.

On the other hand, if $C$ is contained in $(Q/R)\rtimes\left\{1\right\}$, consider the composition of group homomorphisms\[\Phi\colon \pi_1(\Gamma,\mathcal{G},T)\stackrel{\phi}{\longrightarrow}(Q/R)\rtimes_\rho F_n\longrightarrow F_n\] where the second homomorphism is given by projection to the second factor. The preimage $\phi^{-1}(C)$ is a subgroup of $\ker(\Phi)$, whence, by property~(\ref{inheritance: subgroups}), in order to show that $\phi^{-1}(C)$ lies in $\mathfrak{C}$ it suffices to show that $\ker(\Phi)$ lies in $\mathfrak{C}$. We claim that $\ker(\Phi)$ is isomorphic to the fundamental group of a tree of finitely generated abelian groups and hence lies in $\mathfrak{C}$ by Corollary~\ref{corollary: tree of groups}. Equivalently, we claim that $\ker(\Phi)$ acts on a tree with finitely generated abelian point stabilizers and contractible quotient. Since $\pi_1(\Gamma,\mathcal{G},T)$ acts on $X$ with finitely generated abelian point stabilizers, it suffices to show that the quotient $\ker(\Phi)\backslash X$ is a tree. Note that $\pi_1(\Gamma,\mathcal{G},T)/\ker(\Phi)\cong F_n$ and thus we obtain an induced action of the free group $F_n$ on $\ker(\Phi)\backslash X$. As every point stabilizers of the $\pi_1(\Gamma,\mathcal{G},T)$-action on $X$ is contained in $\ker(\Phi)$, the action of $F_n$ on $\ker(\Phi)\backslash X$ is free. We conclude that $\ker(\Phi)\backslash X$ is the universal covering space of the finite graph $F_n\backslash(\ker(\Phi)\backslash X)\cong\Gamma$ and therefore a tree.
\end{proof}

\providecommand{\bysame}{\leavevmode\hbox to3em{\hrulefill}\thinspace}

\end{document}